\documentclass[11pt,letterpaper]{amsart}

\usepackage{fancyhdr}
\usepackage{bm}
\usepackage{hyperref}
\usepackage{graphicx} 
\usepackage{nicematrix}
\usepackage{amsthm,amssymb,amsmath}
\usepackage[T1]{fontenc}
\usepackage[utf8]{inputenc}
\usepackage{hyperref}
\usepackage{lipsum}
\usepackage{mathrsfs}
\usepackage{enumitem}
\usepackage{stmaryrd}
\usepackage{setspace}
\usepackage{yfonts}
\usepackage{float}
\usepackage{mathtools}
\usepackage{parskip}
\usepackage[all,cmtip]{xy}
\usepackage{tikz-cd}
\tikzcdset{row sep/normal=50pt, column sep/normal=50pt}

\newtheorem{lemma}{Lemma}[section]

\newtheorem{theorem}[lemma]{Theorem}
\newtheorem{theorem*}{Theorem}

\newtheorem{example*}[lemma]{Example}

\newtheorem{manualtheoreminner}{Theorem}
\newenvironment{manualtheorem}[1]{%
  \renewcommand\themanualtheoreminner{#1}%
  \manualtheoreminner
}{\endmanualtheoreminner}
\usepackage{blindtext}
\newtheorem{manualquestioninner}{Question}
\newenvironment{manualquestion}[1]{%
  \renewcommand\themanualquestioninner{#1}%
  \manualquestioninner
}{\endmanualquestioninner}

\usepackage[backend=biber, style=numeric, sorting=nyt]{biblatex}
\title{Descend of morphisms of varieties}
\author{\small{Supravat Sarkar}}
\date{}
\begin{document}

\begin{abstract}
   Given varieties $X, Y, W$ and dominant morphisms $\phi:X\to Y$ and $f:X\to W$ such that $f$ is constant on fibres of $\phi$ , we give sufficient conditions to guarantee that $f$ descends to a rational map or a morphism $Y\to W.$ We pay special attention to the case that the ground field has positive characteristic. This extends previous works of Aichinger and Das, who proved similar results for some classes of affine varieties.
\end{abstract}
\maketitle

\section{Introduction}

The goal of this short note is to investigate the following question: 
\begin{manualquestion}{1}
    Given varieties $X, Y, W$ over an algebraically closed field $k$ and dominant morphisms $\phi:X\to Y$ and $f:X\to W$ such that $f$ is constant on fibres of $\phi$, when can we say that $f$ descends to a morphism $h:Y\to W$, or a rational map $h:Y\dashrightarrow W?$ 
\end{manualquestion}A very special case of this question is well-known and used very frequently in birational geometry: when $X,Y$ are normal and $\phi$ is a contraction, that is, $\phi$ is proper and $\phi_*\mathcal{O}_X=\mathcal{O}_Y$, then $f$ always descends to a morphism $h$. So it is instructive to investigate to what extent this question has an affirmative answer in general.

In \cite{aichinger2015function}, this question is studied when $X=\mathbb{A}^n, Y=\mathbb{A}^m$ and $W=\mathbb{A}^1$. They show that in characteristic $0$, $f$ always descends to a rational function $h$ on $Y$. Furthermore, $h$ can be chosen to be a morphism if $\dim \overline{(Y\setminus\phi(X))}\leq \dim Y-2$. In characteristic $p>0$ the same conclusion holds after replacing $f$ by $f^{p^N}$ for some positive integer $N.$

In \cite{das2023regular}, they study this question in characteristic $0$, and obtain the same conclusion for $X$ and $Y$ affine varieties and $W=\mathbb{A}^1$, but with the strong restriction that the coordinate ring of $Y$ is factorial. 

In this note, we study this question in greater generality. We try to make as little assumption as possible. Of course, with no assumption the question has negative answer in general. For example, if $\phi$ is bijective but not an isomorphism (for example $Y$ is cuspidal cubic and $X$ its normalization), then choosing $W=X$ and $f$ to be identity the question has a negative answer if we require $h$ to be a morphism. To avoid this situation we assume $Y$ is normal, when we want $h$ in question to be a morphism. If char $k=p>0$, $\phi$ may be bijective but not birational, for example the map $\mathbb{A}^1\to \mathbb{A}^1$ given by $t\mapsto t^p$. So the question has a negative answer even if we only require $h$ to be a rational map. To remedy this, we replace $f$ by an iterated Frobenius twist of $f$. This is analogous to raising a regular function to its $p^{N}$'th power, which was crucial in \cite{aichinger2015function}.

Recall the definition of iterated Frobenius twist as in {\cite[Section 33.36]{stacks-project}}. Let $k$ be a field of characteristic $p>0$. For $N\in \mathbb{N}$, let $Frob_N:\textrm{Spec }k\to \textrm{Spec }k$ be the map induced by the ring homomorphism $k\to k$ given by $a\mapsto a^{p^N}.$ For a scheme $X$ over $k$, we define the $N$'th Frobenius twist $X^{(p^N)}$ of $X$ by the fibre square 
\begin{center}
\begin{tikzcd}
X^{(p^N)} \arrow[r, ""   ] \arrow[d, ""]
&  X\arrow[d, "", swap ] \\
 \textrm{Spec }k\arrow[r, "Frob_N"]
& |[, rotate=0]|   \textrm{Spec }k
\end{tikzcd}
\end{center}

We have the  iterated relative Frobenius morphism $X\to X^{(p^N)} $ over $k$, induced by the iterated abosolute Frobenius morphism of $X.$

We recall a few more standard notations. For a variety $X$, we denote its ring of global regular functions and field of rational functions by $\mathcal{O}(X)$ and $K(X)$ respectively. For a point $x\in X$, $\mathcal{O}_x$ denotes the local ring at $x$, and $\mathfrak{m}_x, k(x)$ denotes its maximal ideal and residue field, respectively. For a field extension $L/K$, we say $K$ is \textit{purely inseparably closed} in $L$ if any $f\in L$ which is algebraic and purely inseparable over $K$ lies in $K.$ If char $K=0$, then $K$ is always purely inseparably closed in $L$. If char $K=p>0$, then $K$ is purely inseparably closed in $L$ iff for every $f\in L\setminus K$, we have $f^p\not\in K.$

Now we can state our result.
 \begin{manualtheorem}{1}\label{A}
     Let $X, Y, W$ be varieties over an algebraically closed field $k$, and let $\phi:X\to Y$ and $f:X\to W$ be dominant morphisms. Suppose $f$ is constant on each fibre of $\phi$ over closed points of $Y$. For $N\in \mathbb{N}$, define varieties $W_N$ and morphisms $F_N:W\to W_N$ as follows: 
     
     If $K(Y)$ is purely inseparably closed in $K(X)$ let $W_N=W$ and $F_N$ the identity map, otherwise let $W_N= W^{(p^N)}$ and $F_N:W\to W_N$ the $N$'th iterated relative Frobenius morphism, where char $k=p$.
     
     Then the following holds:
     \begin{enumerate}
         \item There is $N\in \mathbb{N}$ and a dominant rational map $h:Y\dashrightarrow W_N$ such that $F_N\circ f=h\circ \phi$ as rational maps $X\dashrightarrow W_N$.
         \item Suppose furthermore $Y$ is normal, and $\dim \overline{(Y\setminus\phi(X))}\leq \dim Y-2$. Then the following are equivalent: 
         \begin{enumerate}
         \item There is $N\in \mathbb{N}$ and a morphism $h:Y\rightarrow W_N$ such that $F_N\circ f=h\circ \phi$,
         \item There is a continuous map in Zariski topology $h:Y\rightarrow W$ such that $f=h\circ \phi$ as continuous maps,
             \item There are open covers $Y=\cup_iU_i, W=\cup_iV_i$, with each $V_i$ affine and $f(\phi^{-1}U_i)\subseteq V_i.$
             
         \end{enumerate}
If either $W$ is affine or $\phi$ is a closed map or a surjective open map, then these equivalent conditions hold.
     \end{enumerate}
 \end{manualtheorem}

In \S 2 and \S3, we prove this theorem. In \S 4, we give an application of this theorem. We also give an example to illustrate that the equivalent conditions in Theorem \ref{A} (2) might not always hold, even if $k=\mathbb{C}$, $X,Y$ are smooth and $\phi$ is surjective.
 \section{Proof of $(1)$}

 Now we prove Theorem \ref{A}. We first prove part $(1)$. For ease of understanding, we divide the proof into $2$ steps.

     \textbf{Step 1:} We reduce to the case $W=\mathbb{A}^1$.

     Let $V\subseteq W$ be a nonempty open affine. So, $f^{-1}V$ is a dense open set in $X$. Also, $f^{-1}V=\phi^{-1}(\phi(f^{-1}V))$ as sets, as $f$ is constant on fibres of $\phi$ over closed points. As $\phi$ is dominant, $\phi(f^{-1}V)$ is a dense constructible set in $Y$, so there is a nonempty open $U\subseteq Y$ with $U\subseteq \phi(f^{-1}V)$. We have $f(\phi^{-1}U)\subseteq V.$ So, replacing $(X, Y, W)$ by $(\phi^{-1}U, U, V)$, we may assume $W$ is affine. Embedding $W$ in some $\mathbb{A}^r$ and considering separately the components of $f$, we can assume $W=\mathbb{A}^1$.
     
     So, $f$ is a regular function on $X$, and if char $k=p>0$, then $(\mathbb{A}^1)^{(p^N)}=\mathbb{A}^1$ and $F_N:\mathbb{A}^1\to \mathbb{A}^1$ is given by $t\mapsto t^{p^N}.$  We identify $K(Y)$ as a subfield of $K(X)$ via $\phi^*$. It suffices to show the following:
     \begin{enumerate}
         \item[$(i)$] If $K(Y)$ is purely inseparably closed in $K(X)$, then $f\in K(Y)$,
         \item[$(ii)$] If char $k=p>0$, the $f^{p^N}\in K(Y)$ for some positive integer $N$.
     \end{enumerate}

\textbf{Step 2:} We prove $(i)$ for char $k=0$, and $(ii)$. Clearly, if char $k=p>0$, $(i)$ would follow from $(ii)$. So, this would complete the proof of $(1)$.

Let $\xi$ be the generic point of $Y$, and let $\eta$ be a closed point of $\phi^{-1}(\xi)$. So, $k(\eta)$ is a finite extension of $k(\xi)=K(Y)$. Let $Z\subseteq X$ be the closure of $\eta$, with reduced induced structure. So, $\dim Z=\dim Y$ and $\phi|_Z:Z\to Y$ is dominant. If $f|_Z\in K(Y)$, then $f\in K(Y)$ by the following argument: choose dense open $U\subseteq Y, h\in \mathcal{O}(U)$ such that $U\subseteq\phi(Z)$ and $h\circ \phi=f$ on $Z\cap \phi^{-1}(U)$. Given any $x\in \phi^{-1}(U)$, there is $z\in Z\cap \phi^{-1}(U)$ such that $\phi(z)=\phi(x)$. As $f$ is constant on each fibre of $\phi$, we have $$f(x)=f(z)=h\phi(z)=h\phi(x).$$ So, $h\circ \phi=f$ on $\phi^{-1}U$. Hence $f=\phi^*h\in K(Y).$

If char $k=p>0$, the same argument shows $f^{p^N}\in K(Y)$ if $(f|_Z)^{p^N}\in K(Y)$. So, replacing $X$ by $Z$, we can assume $\dim X=\dim Y.$ So, $\phi$ is generically finite.

Let $L$ be the normal closure of $K(X)$ over $K(Y)$. There is a variety $\Tilde{X}$ with $K(\Tilde{X})=L$, and a dominant morphism $\Tilde{X}\to X$ inducing the inclusion $K(X)\subset L.$ Replacing $X$ by $\Tilde{X}$ and $f$ by pullback of $f$ to $\Tilde{X}$, we may assume that $K(X)$ is normal over $K(Y)$.

Let $G$ be the finite group of automorphisms of $K(X)$ over $K(Y)$. So, the invariant subfield $K(X)^G$ is the purely inseparable closure of $K(Y)$ in $K(X)$. Hence it suffices to show $f\in K(X)^G.$ That is, if $g:X\dashrightarrow X$ is a birational automorphism of $X$ over $Y$ corresponding to an element of $G$, we need to show  $f\circ g=f$ in $K(X)$. But this follows from our hypothesis that $f$ is constant of fibres of $\phi.$
\section{Proof of $(2)$:}

 We first prove the following key lemma.
 \begin{lemma}\label{extend}
     Let 
     \begin{center}
    \begin{tikzpicture}[>=stealth, node distance=2.2cm, auto]
  \node (X) at (0,2) {$X$};
  \node (Y) at (0,0) {$Y$};
  \node (W) at (4,0) {$W$};

  \draw[->, thick] (X) -- node[midway, above, sloped] {$f$} (W);

  \draw[->, thick] (X) -- node[midway, left] {$\varphi$} (Y);

  \draw[->, dashed, thick] (Y) -- node[midway, below] {$h$} (W);
\end{tikzpicture}
\end{center}
be a commutative diagram of varieties over an algebraically closed field $k$, $f$ and $\phi$ are dominant morphisms, $h$ a dominant rational map. Suppose $Y$ is normal, $W$ is affine, and $\dim \overline{(Y\setminus\phi(X))}\leq \dim Y-2$. Then $h$ is a morphism.
 \end{lemma}
 \begin{proof}
     Embedding $W$ in some $\mathbb{A}^r$, it suffices to show each component of $h$ is a morphism to $\mathbb{A}^1$. So, we can assume without loss of generality $W=\mathbb{A}^1$. So, $f\in \mathcal{O}(X)$ and $h\in K(Y)$.

     By normality of $Y$, it suffices to show $h\in \mathcal{O}_y$ for each codimension $1$ point $y\in Y$. As $\dim \overline{(Y\setminus\phi(X))}\leq \dim Y-2$, we have $y\in \phi(X)$. So, there is a codimension $1$ point $x\in X$ with $\phi(x)=y.$ We have extension of function fields $\phi^*:K(Y)\hookrightarrow K(X)$, inducing local extension $\mathcal{O}_y\hookrightarrow \mathcal{O}_x$, and $\phi^*h\in \mathcal{O}_x. $ If $h\not\in \mathcal{O}_y$, then $\frac{1}{h}\in \mathfrak{m}_y$ as $\mathcal{O}_y$ is a discrete valuation ring, hence $\frac{1}{\phi^*h}\in \mathfrak{m}_x.$ This contradicts $\phi^*h\in \mathcal{O}_x. $ So, $h\in \mathcal{O}_y$.
 \end{proof}
We are now ready to prove $(2)$. 

\underline{$(a)\implies(b):$} Follows as $F_N$ is a homeomorphism by {\cite[Lemma 33.36.6]{stacks-project}}.

\underline{$(b)\implies(c):$} Take $W=\cup_iV_i$ to be any open affine cover and let $U_i=h^{-1}(V_i)$, which is open as $h$ is continuous.

\underline{$(c)\implies(a):$} Note that for all $i$ and $N$, $F_N(V_i)$ is either $V_i$ or $V_i^{(p^N)}$, hence is affine. By Theorem \ref{A} $(1)$ we have a dominant rational map $h:Y \dashrightarrow W_N$ for some $N$ such that $F_N\circ f=h\circ \phi$. By Lemma \ref{extend}, applied to $(\phi^{-1}U_i, U_i, F_N(V_i), F_N\circ f)$ instead of $(X,Y, W, f)$, we get $h|_{U_i}$ is a morphism for all $i$. Hence $h$ is a morphism.

 To prove the last statement, note that if $W$ is affine then $(c)$ holds by taking $V_i=W$. If $\phi$ is a closed map or a surjective open map, then $\phi$ is a quotient map in Zariski topology, hence $(b)$ holds.
\section{Application and an example}

We have the following application our Theorem \ref{A}.
\begin{theorem}\label{app}
     Let $X, Y, W$ be varieties over an algebraically closed field $k$, with $Y$ normal. Let $\phi:X\to Y$ and $f:X\to W$ be dominant morphisms, with $\phi$ a closed map. Assume $K(Y)$ is purely inseparably closed in $K(X)$. Suppose $f$ is constant on each fibre of $\phi$ over closed points of $Y$. There is a morphism $h:Y\rightarrow W$ such that $f=h\circ \phi$.
\end{theorem}
\begin{proof}
    As $\phi$ is dominant and closed, $\phi$ is surjective. As $\phi$ is closed, we are done by Theorem \ref{A}(2), noting that $W=W_N$ by assumption. 
\end{proof}
Note that Theorem \ref{app} is a generalization of the well-known fact about contraction mentioned in the beginning of the introduction, as by the same proof as in {\cite[Example 2.1.12]{lazarsfeld2017positivity}}, $K(Y)$ is algebraically closed in $K(X)$ if $\phi$ is a contraction.

Of course, one can give more applications of Theorem \ref{A} like removing the factorial assumption of $Y$ from {\cite[Theorem 3.1]{das2023regular}} and {\cite[Theorem 3.3]{das2023regular}}, but we do not state them here as we feel that those results follows easily from more classical results like Zariski's main theorem.

 Now we give an example to illustrate that the equivalent conditions in Theorem \ref{A} (2) might not always hold, even if $k=\mathbb{C}$, $X,Y$ are smooth and $\phi$ is surjective. 

 Let $k=\mathbb{C}$, $Y=\mathbb{P}^2$. Let $O\in Y$ be a closed point. Let $X'\xrightarrow{\phi'}Y$ be the blow up of $Y$ at $O$, with exceptional divisor $E$. Choose a closed point $P\in E$ and let $Z\xrightarrow{g} X'$ be the blow up of $X'$ at $P$, with $E_2$ the exceptional divisor and $E_1$ the strict transform of $E$ in $Z$. We have $E_1^2=-2$, so there is a contraction $p:Z\to W$ to a normal projective surface $W$ contracting $E_1$ to a point $Q\in W$ where $W$ has $A_1$-singularity. Let $X=Y\setminus P$, and $\phi:X\to Y$ the restriction of $\phi'.$ As $g$ is an isomorphism over $X$, $p$ induces a morphism $f:X\to W$, which is constant $E\setminus P$, the only nonsingleton fibre of $\phi$. Nevertheless, $f$ does not descend to a morphism $Y\to W$, as the birational map $h=f\circ\phi^{-1}:Y\dashrightarrow W$ cannot be defined at $O$.

 \begin{center}

\begin{tikzpicture}[>=stealth, node distance=2cm, auto]
  \node (Z) at (2,3) {$Z$};
  \node (X1) at (0,1.5) {$X'$};
  \node (X) at (2,1.5) {$X$};
  \node (Y) at (0,0) {$Y$};
  \node (W) at (4,1.5) {$W$};

  \draw[->, thick] (Z) -- node[left] {$g$} (X1);
  \draw[->, thick] (Z) -- node[above right] {$p$} (W);

  \draw[->, thick] (X) -- node[above] {$f$} (W);
  \draw[][->, thick] (X) -- node[above] {$\supseteq$} (X1);

  \draw[->, thick] (X1) -- node[left] {$\varphi'$} (Y);
  \draw[->, thick] (X) -- node[right] {$\varphi$} (Y);

\end{tikzpicture}

 \end{center}

\printbibliography
\vspace{40pt}
\begin{flushleft}
{\scshape Fine Hall, Princeton University, Princeton, NJ 700108, USA}.

{\fontfamily{cmtt}\selectfont
\textit{Email address: ss6663@princeton.edu} }
\end{flushleft}
\end{document}